\documentclass[12pt]{amsart}
\usepackage{amsmath,amssymb,amsfonts}
\usepackage[colorlinks=true,linkcolor=red,citecolor=black]{hyperref}
\newtheorem{theorem}{Theorem}[section]
\newtheorem{proposition}{Proposition}[section]
\newtheorem{corollary}{Corollary}[section]

\newtheorem{remark}{Remark}[section]

\newtheorem{definition}{Definition}[section]

\newcommand\myenum[1]

\begin{document}
\title[Some results on the class of unbounded Dunford-Pettis operators]{Some results on the class of unbounded Dunford-Pettis operators}
\author{N. Hafidi}
\author{ J. H'michane}

\address{Noufissa Hafidi, Universit\'{e} Moulay Ismail, Facult\'{e} des Sciences, D\'{e}partement de Math\'{e}matiques, B.P. 11201 Zitoune, Mekn\`{e}s, Morocco.}

\address{Jawad H'michane, Universit\'{e} Moulay Ismail, Facult\'{e} des Sciences, D\'{e}partement de Math\'{e}matiques, B.P. 11201 Zitoune, Mekn\`{e}s, Morocco, And, Gosaef, Faculte des Sciences de Tunis, Manar II-Tunis (Tunisia)}
\email{hm1982jad@gmail.com}

\begin{abstract}
	
We introduce and study the  class of unbounded Dunford-Pettis operators. As consequences, we give basic properties and derive interesting results about the duality, domination problem and  relationship with other known classes of operators.
\color{black}
\end{abstract}
\maketitle
\section{Introduction}
Along this paper the term operator means a bounded linear mapping.\\
The notion of unbounded order convergence was appeared in \cite{NAK} and was studied in many several papers \cite{GAO1,GAO2, WIC}. Our interest focus on unbounded norm convergence in Banach lattices, which is a recent extension of unbounded order convergence investigated in many works such as \cite{DOT,KD,KADIC}.

 The class of Dunford-Pettis operators is among the most extensively studied  in history of operators acting on Banach lattices literature. In this note we present a generalization of the  Dunford-Pettis operators, by introducing and studying a new classes of operators. Indeed, using the unbounded norm convergence and unbounded norm topology in Banach lattices, we introduce
 the concept of sequentailly unbounded Dunford-Pettis operators (abb, $\sigma$-un-Dunford-Pettis) and we investigate about the following facts:
\begin{itemize}
\item Properties of this class of operators.
\item Domination property.
\item Duality property.
\item Relationships with other known classes of operators.
\end{itemize}

\section{Preliminaries}

Let us recall that a linear mapping $T:X\longrightarrow Y$ between two Banach spaces $X$ and $Y$  is said to be Dunford-Pettis, whenever $x_n{\overset{w}{\longrightarrow}} 0$ in $X$ implies $\|T(x_n)\|\rightarrow0$. Alternatively, $T$ is  Dunford-Pettis if and only if $T$ carries relatively weakly compact sets into compact sets. A linear mapping $T:X\longrightarrow Y$ between two Banach spaces $X$ and $Y$  is said to be weakly compact, whenever $T$ carries the closed unit ball of $X$ to a weakly relatively compact subset of $Y$. An operator $T:E\longrightarrow X$ from a Banach lattice $E$ to a Banach space $X$ is said to be  M-weakly compact if  $\|T(x_n)\|\rightarrow0$ for each norm bounded disjoint sequence $(x_n)$  of $E$. An operator $T:E\longrightarrow X$ from a Banach lattice $E$ to a Banach space $X$ is order weakly compact whenever $T[0,x]$ is a relatively weakly compact subset of $X$ for each $x\in E^{+}$, equivalently,  $\|T(x_n)\|\rightarrow0$ for each order bounded disjoint sequence $(x_n)$  of $E$.

	To state our results, we need to fix some notations and recall some definitions. A Banach lattice is a Banach space $(E,\Vert \cdot \Vert )$ such that $E$ is a vector lattice and its norm satisfies the following property: for each $x,y\in E$ such that $|x|\leq |y|$, we have $\Vert x\Vert \leq \Vert y\Vert $. If $E$ is a Banach lattice, its topological dual $E^{\prime }$, endowed with the dual norm, is also a Banach lattice. A Banach lattice $E$ is order continuous, if for each net $(x_{\alpha })$ such that $x_{\alpha }\downarrow 0$ in $E$, the sequence $(x_{\alpha })$ converges to $0$ for the norm $\Vert \cdot \Vert $, where the notation $x_{\alpha }\downarrow 0$ means that the sequence $(x_{\alpha })$ is decreasing, its infimum exists and $\inf(x_{\alpha })=0$.
	
A vector lattice $E$ is Dedekind $\sigma$-complete if every majorized countable nonempty subset of $E$ has a supremum.

A Banach lattice $E$ is said to be KB-space, if every increasing norm bounded sequence of $E^+$ is norm convergent.

A nonzero element $x$ of a vector lattice $E$ is discrete if the order ideal generated by $x$ equals the vector subspace generated by $x$. The vector lattice $E$ is atomic, if it admits a complete disjoint system of discrete elements.
	
The lattice operations in $E$ are  weakly  sequentially continuous, if the sequence $(|x_{n}|)$  converges to $0$ in the weak topology, whenever the sequence $(x_{n})$  converges weakly  to $0$ in $E$.

 A vector subspace $G$ of a Riesz space $X$ is majorizing $E$, whenever for each $x\in E$ there exists some $y\in G$ with $x\leq y$, equivalently, if for each $x\in E$ there exists some $y\in G$ with $y\leq x$.

 A band $B$ in a Riesz space $X$ that satisfies $X=B\oplus B^d$ is called a projection band, where $B^d$ stands for the disjoint complement of $B$.

 A vector $e>0$ in a Riesz space $X$ is said to be order unit whenever for each $x\in X$ there exists some $\lambda >0$ satisfying $x\leq \lambda e$.

 A positive element $e$ of a normed Riesz space is called quasi-interior  point  if the ideal $E_{u}$ generated by $u$ is norm dense.

 A Banach lattice has the Schur (resp. positive Schur) property whenever  $x_n{\overset{w}{\longrightarrow}} 0$ (resp. $0\leq x_n{\overset{w}{\longrightarrow}} 0$ ) implies  $\|x_n\|\longrightarrow0$.
	
A net $(x_{\alpha})$ in  a Riesz space $X$ is \textbf{unbounded order convergent} (abb. uo-convergent) if, $ |x_{\alpha}-x|\wedge u$ converges  to $0$ in order (abb. $x_{\alpha}{\overset{uo}{\longrightarrow}} 0$), for each $u\in X^+$.
		
A net $(x_{\alpha})$ in  a Banach lattice $E$ is \textbf{unbounded norm convergent} (abb. un-norm convergent) if, $\| |x_{\alpha}-x|\wedge u  \| \longrightarrow 0 $ (abb. $x_{\alpha}{\overset{un}{\longrightarrow}} 0$), for each $u\in E^+$. Note that the norm convergence implies the un-norm convergence. We can easily check, that un-convergence coincides with norm convergence on a Banach lattice with order unit, in particular for order bounded nets.
	
If $T$ is a  linear mapping from a Banach space $X$ into a Banach space $Y$ then, its adjoint operator $T^{\prime}$ is defined from $Y^{\prime}$ into $X^{\prime}$ by $T^{\prime}(f)(x) = f (T (x))$	for each $f\in Y^{\prime}$ and for each $x\in X$. We refer the reader to \cite{AB} for unexplained terminology on Banach lattice theory.

\section{Main results}
We start by the following definitions:

\begin{definition}
	An operator  $T$ from a Banach space $X$ into a Banach lattice $F$  is said to be sequentially unbounded Dunford-Pettis (abb. $\sigma$-un-Dunford-Pettis) whenever   $(T(x_{n}))$ is un-norm null, for each weak-null sequence $(x_{n})$ in $E$.
\end{definition}

\begin{definition}
	An operator  $T$ from  a Banach space $X$ into a Banach lattice $F$  is said to be unbounded Dunford-Pettis (abb. un-Dunford-Pettis) whenever  $(T(x_{\alpha}))$ is un-norm null for every weak-null net $(x_{\alpha})$ in $X$.
\end{definition}
Our research interests in this paper focus on the $\sigma$-un-Dunford-Pettis operators. The class of  $\sigma$-un-Dunford-Pettis operators will be denoted by $DP_{\sigma-un}(X,F)$.

As an immediate consequence of  Proposition 6.2 \cite{DOT}, we have the following result:

\begin{proposition} \label{3.1}
	Let $E$ be a Banach lattice. Then the identity operator $Id_{E}:E\longrightarrow E$ is $\sigma$-un-Dunford-Pettis if and only if  $E$ is order continuous and atomic.
\end{proposition}

We recall from \cite{KADIC} that a subset $A$ of a Banach lattice $E$ is said to be \textbf{ un-compact} (respectively, \textbf{sequentially un-compact}), if every net $(x_\alpha)$ (respectively, every  sequence $(x_n)$) in $A$ has  a subnet (respectively, subsequence) which is un-norm convergent.

In the following result, we give an important characterization of $\sigma$-un-Dunford-Pettis operators.

\begin{proposition} \label{DEF}
	Let $X$ be a Banach space, $F$ be a Banach lattice and $T$ an operator from $X$ into $F$. Then the following statements are equivalent:
	\begin{enumerate}
		\item $T$ is $\sigma$-un-Dunford-Pettis.
		\item $T(A)$ is relatively  sequentially un-compact, for each  weakly compact set $A$ of $X$.
	\end{enumerate}
	
\end{proposition}

\begin{proof}
$(1)\Longrightarrow	(2)$ Let $A\subset X$ be a  weakly compact set, we will show that $T(A)$ is relatively sequentially un-compact. To this end, let $y_n \in T(A)$, so that $y_n=T(x_n)$ for some  $(x_n)\in A$. Now, as $A$ is  weakly compact, by the Eberlian-Smulian Theorem there exists $(x_{n_{k}})$ a subsequence of $(x_n)$ such that $(x_{n_{k}}){\overset{w}{\longrightarrow}} y$ for some $y\in X$, we may assume that $y=0$. As $T$ is  $\sigma$-un-Dunford-Pettis, then $T(x_{n_{k}}){\overset{un}{\longrightarrow}} 0$, this shows that $T(A)$ is a relatively sequentially un-compact set of $F$.

$(2)\Longrightarrow	(1)$ Let $(x_n)$ be a weak-null sequence in $X$ and $(n_k)_k$ is any increasing sequence of positive integers. Then notice that  $(x_{n_k})_k$ is still weak-null sequence, hence if  we set $A=\{ x_{n_k},  k\in\mathbb{N} \}\cup \{0\} $ clearly  $A$ is a weakly compact set of $X$ and by assumption  $T(A)$ is a relatively sequentially un-compact set in $F$. Hence,  there exists $(T(x_{n_{k'}}))$ a subsequence of $(T(x_{n_k}))$ such that $T(x_{n_{k'}}){\overset{un}{\longrightarrow}} y$ for some $y\in X$, uniqueness of limit implies that $y=0$, so that  $T(x_{n_{k'}}){\overset{un}{\longrightarrow}} 0$. Thus,  given any subsequence $(T(x_{n_k}))$ of  $ T(x_{n})$, there exists a subsequence of this  subsequence  which is un-norm null. This implies that $T(x_{n}){\overset{un}{\longrightarrow}} 0$ and shows that $T$ is  $\sigma$-un-Dunford-Pettis.
	\end{proof}

\begin{remark}	
	Each Dunford-Pettis  operator is $\sigma$-un-Dunford-Pettis, but the converse does not hold in general. Indeed, the identity operator $Id_{c_0}$ of the Banach lattice $c_0$  is $\sigma$-un-Dunford-Pettis ( since  $c_0$ is order continuous and atomic) but fails to be Dunford-Pettis.
\end{remark}

\begin{remark}	
	We can check easily that the set of all $\sigma$-un-Dunford-Pettis operators between a Banach space $X$ and a Banach lattice $F$  is a closed linear subspace of $L(X,F)$.
\end{remark}

For a further results, we need to focus on other concepts around unbounded norm convergence and unbounded norm topology. Namely, it seems reasonable to introduce the notion of continuity related to these concepts and which will be used in the next results.
\begin{definition}\cite{CHEN}
	Let $E$ and $F$ be two Banach lattices.
	\begin{enumerate}
		\item An operator  $T:E\longrightarrow F$ is called unbounded norm continuous (abb. un-continuous), if $x_\alpha{\overset{un}{\longrightarrow}} 0$ in $E$ implies  $T(x_\alpha){\overset{un}{\longrightarrow}} 0$ in $F$.
		\item  An operator  $T:E\longrightarrow F$ is called  $\sigma$-unbounded norm continuous (abb. $\sigma$-un-continuous), if $x_n{\overset{un}{\longrightarrow}} 0$ in $E$ implies  $T(x_n){\overset{un}{\longrightarrow}} 0$ in $F$.
	\end{enumerate}
\end{definition}

\begin{remark}
	We should mention that unbounded continuity is different from norm continuity. Indeed, we consider the canonical injection $i:c_0\longrightarrow \ell^{\infty}$ which is a norm continuous operator, but by considering the Example 2.6 \cite{DOT} we can check that the canonical injection $i:c_0\longrightarrow \ell^{\infty}$ fails to be un-continuous.	
\end{remark}

\begin{proposition}
	Let $E$ and $F$ be two Banach lattices.
	If $T:E\longrightarrow F$ is a lattice  homomorphism such that its range is norm-dense or majorizing or projection band, then $T$ is an un-continuous operator. In particular, if $T:E\longrightarrow F$ is an onto lattice homomorphism, then $T$ is an un-continuous operator.
	
\end{proposition}

\begin{proof}
	Let $(x_{\alpha})$ be a net in $E$ such that $x_{\alpha}{\overset{un}{\longrightarrow}} 0$ and $v\in (T(E))^{+}=T(E^{+})$ (since $T$ is lattice homomorphism), then there exists $u\in E^+$ such that $T(u)=v$, in particular we have $|x_{\alpha}|\wedge u{\overset{\|.\|}{\longrightarrow}} 0$. On the other hand, we have the equality  $$|T(x_{\alpha})|\wedge v=T(|x_{\alpha}|)\wedge T(u)=T(|x_{\alpha}|\wedge u){\overset{\|.\|}{\longrightarrow}} 0 \hspace{0.3cm} \text{for all}\hspace{0.3cm}  u\in E^+  .$$
	Thus, $T(x_{\alpha}){\overset{un}{\longrightarrow}} 0$ in $T(E)$. As $T$ is a lattice homomorphism, we have that $T(E)$ is a sublattice of $F$ and hence the result follows from  Theorem 4.3 \cite{KD}.
\end{proof}

Our next interest is to determine whether the set of $\sigma$-un-Dunford-Pettis operators forms a two-sided ideal in the class of all operators. Note that the class of $DP_{\sigma-un}(X,F)$ is a right ideal, but  unfortunately it is not a left ideal. Indeed, if we consider the composed operator; $$c_0\overset{Id_{c_0}}{\longrightarrow} c_0\overset{i}{\longrightarrow} \ell^{\infty}$$  Where $i:c_0\longrightarrow \ell^{\infty}$ is the canonical injection, we check easily  that  $i\circ Id_{c_0}=i$ is not $\sigma$-un-Dunford-Pettis. However,  $Id_{c_0}$ the identity operator of $c_0$ is $\sigma$-un-Dunford-Pettis.\\
Nevertheless, we have the following result;

\begin{proposition}\label{4.2}
	Let $X$ be a Banach space, and $G$, $F$ be  Banach lattices.
 If $T:G\longrightarrow F$  is a $\sigma$-un-continuous operator, then for each $\sigma$-un-Dunford-Pettis operator $S:X\longrightarrow G$ the composed operator $T\circ S$ is $\sigma$-un-Dunford-Pettis.
\end{proposition}

\begin{proof}	
Let $(x_n)$ be weak-null sequence in $E$, since $S$ is a $\sigma$-un-Dunford-Pettis operator, we have that $S(x_n){\overset{un}{\longrightarrow}} 0 $, using the fact that $T$ is $\sigma$-un continuous, it follows that $T\circ S(x_n){\overset{un}{\longrightarrow}} 0 $, and therefore $T\circ S$ is $\sigma$-un-Dunford-Pettis.
\end{proof}		

Note that an order weakly compact operator is not in general $\sigma$-un-Dunford-Pettis. Indeed, the canonical injection  $ i:c_0\longrightarrow \ell^\infty$ is order weakly compact but not $\sigma$-un-Dunford-Pettis, (by Example 2.6 \cite{DOT} $i(e_n)=e_n {\overset{un}{\nrightarrow}} 0$ in $\ell^\infty$).

\begin{proposition}\label{3.3.1}
	Let $E$ and $F$ be two Banach lattices and $T:E\longrightarrow F$ an order bounded operator. If $T$ is $\sigma$-un-Dunford-Pettis, then $T$ is order weakly compact.
\end{proposition}

\begin{proof}
 Let $(x_n)_n$ be an order bounded weak-null sequence in $E^{+}$, since $T$ is  $\sigma$-un-Dunford-Pettis it follows that  $T(x_n){\overset{un}{\longrightarrow}} 0$, by order boundedness of $T$ we have $(T(x_n))_n$ is an order bounded sequence and then $T(x_n){\overset{\|.\|}{\longrightarrow}} 0$. Hence, it follows from Corollary 3.4.9\cite{Mey} that $T$ is an order weakly compact operator.
\end{proof}

\begin{remark}	
The condition `` $T:E\longrightarrow F$ an order bounded operator" is essential in Proposition \ref{3.3.1}. Indeed, we choose a non-weakly compact operator $T:E\longrightarrow F$ such that $F$ is a non-reflexive  atomic order continuous Banach lattice (like $c_0$) and $E$ is a Banach lattice with order unit. Since $F$ is an atomic order continuous Banach lattice, then the operator $T$ is $\sigma$-un-Dunford-Pettis. On the other hand, $E$ has order unit and $T$ is a non-weakly compact operator implies that $T$ is necessary not order weakly compact and it is also not order bounded operator. Otherwise, since $B_E$ is order bounded, where $B_E$ is the closed unit ball of $E$, then $T(B_E)$ is order bounded in $F$ which is order continuous, it follows from Theorem 4.9 \cite{AB} that $T(B_E)$ is weakly compact, that is $T$ is a weakly compact operator and this is a contradiction.
\end{remark}

\begin{proposition} \label{3.3.}
	Let $E$ and $F$ be Banach lattices such that $E$ has quasi-interior point and the lattice operations of $E$ are weakly sequentially continuous. If $T:E\longrightarrow F$ is a lattice homomorphism with norm-dense range, then the following statements are equivalent:
	\begin{enumerate}
		\item $T$ is order weakly compact.
		
		\item $T$ is $\sigma$-un-Dunford-Pettis.
	\end{enumerate}
\end{proposition}

\begin{proof}
	$(1)\Rightarrow (2)$ Let $(x_n)$ be a weak-null sequence and $e$ a quasi-interior point of $E$, since the  lattice operations of $E$ are weakly sequentially continuous, we have  $|x_n|\wedge e$ is weak-null and it is a positive order bounded sequence. Now, as $T$ is order weakly compact it follows   from  Corollary 3.4.9 \cite{Mey} that  $T(|x_n|\wedge e){\overset{\|.\|}{\longrightarrow}} 0$, on the other hand, $T$ is a lattice homomorphism, hence we have the following equality  $$|T(x_n)|\wedge T(e)=T(|x_n|)\wedge T(e)=T(|x_n|\wedge e).$$  So, $|T(x_n)|\wedge Te{\overset{\|.\|}{\longrightarrow}} 0$. Now, by [Exercise 11 Sec.15 \cite{AB}] $T(e)$ is a quasi-interior point of $F$ and from Lemma 2.11 \cite{DOT} we infer that $T(x_n)$ is un-null, therefore $T$  is $\sigma$-un-Dunford-Pettis.\\
	$(2)\Rightarrow (1)$ Since each lattice homomorphism $T:E\longrightarrow F$ is order bounded, the implication follows from Proposition \ref{3.3.1}.
\end{proof}

We should mention that the domination problem for the class  of $\sigma$-un-Dunford-Pettis operators  is not verified. Indeed, we consider the example mentioned in \cite{AB}, let $S,T: L_{1}[0,1]\longrightarrow \ell^{\infty}$ be two positive operators defined by

	$$ \begin{array}{lrcl}
S : & L_{1}[0,1] & \longrightarrow &  \ell^{\infty} \\
&f & \longmapsto &  \left( \displaystyle\int_{0}^{1}f(x) r_{1}^{+}(x) dx,\int_{0}^{1}f(x) r_{2}^{+}(x) dx,...\right) \end{array}  $$
and
$$ \begin{array}{lrcl}
T : & L_{1}[0,1] & \longrightarrow &  \ell^{\infty}\\
& f & \longmapsto &  \left( \displaystyle\int_{0}^{1}f(x) dx,\int_{0}^{1}f(x)  dx,...\right)  \end{array}  $$

Where $(r_{n})$ is the sequence of Rademacher functions on $[0,1]$, we easily check that $ 0\leq S\leq T$ and that $T$ is a compact, hence it is $\sigma$-un-Dunford-Pettis operator. On the other hand, we have $r_{n}{\overset{w}{\longrightarrow}} 0$ in $L_{1}[0,1]$. It is observed also that $\|S(r_{n})\|_{\infty}\geq \int_{0}^{1}r_{n}(x) r_{n}^{+}(x) dx =\frac{1}{2}$. Now, we remind that norm convergence and unbouded norm convergence coincide in Banach lattices with order unit. Therefore, $ S(r_{n}) {\overset{un}{\nrightarrow}} 0$ and this shows that $S$ is not $\sigma$-un-Dunford-Pettis operator.

But  under some conditions, we find positive solution as, the next result shows.

\begin{proposition}
	Let $E$ and $F$ be two Banach lattices and $S,T:E\longrightarrow F$ be two positive operators such that $0\leq S\leq T$. The class of positive $\sigma$-un-Dunford-Pettis operators satisfies the domination property if one of the following statements is valid:
\begin{enumerate}
\item $T$ is a lattice homomorphism.
\item The lattice operations of $E$ are weakly sequentially continuous.
\end{enumerate}
\end{proposition}

\begin{proof}
\begin{enumerate}
\item Let $S$ and $T$ be two operators from $E$ into $F$ such that $0\leq S \leq T$ and  $T$ is $\sigma$-un-Dunford-Pettis and let $(x_n)$ be a weak-null sequence, we have that $|S(x_n)|\wedge u\leq T(|x_n|)\wedge u$. As, the lattice  operations of $E$ are weak sequentially continuous, it follows that $|x_n|{\overset{w}{\longrightarrow}} 0$ and as $T$ is $\sigma$-un-Dunford-Pettis,  we have  $T(|x_n|)\wedge u=|T(|x_n|)|\wedge u{\overset{\|.\|}{\longrightarrow}} 0$. So, we can conclude that $S$ is a $\sigma$-un-Dunford-Pettis operator.
\item	Let $S$ and $T$ be two operators from $E$ into $F$ such that $0\leq S \leq T$ and  $T$ is $\sigma$-un-Dunford-Pettis and let $(x_n)$ be a weak-null sequence in $E$, since $T$ is $\sigma$-un-Dunford-Pettis, it follows that  $|T(x_n)|\wedge u{\overset {\|.\|}{\longrightarrow}} 0$  for all $u\in F^+$. On the other hand, $T$ is a lattice homomorphism, so $|T(x_n)|=T(|x_n|)$ and $|S(x_n)|\wedge u\leq S(|x_n|)\wedge u \leq T(|x_n|) \wedge u= |T(x_n)|\wedge u\overset{\|.\|}{\longrightarrow} 0$, for each $u\in F^+$. Therefore $S$ is a $\sigma$-un-Dunford-Pettis operator.
\end{enumerate}
\end{proof}

We are in position to give a new result about M-weakly compact operators in terms of sequentially un-compact sets.

\begin{proposition} \label{MW}
	Let $E$ be a Banach lattice and $Y$ be a Banach space. If $T:E\longrightarrow Y$ is an M-weakly compact operator, then  $T(A)$ is a relatively compact set in $Y$ for each  relatively sequentially un-compact set $A$ in $E$.
\end{proposition}

\begin{proof}
	Let $A$ be a relatively sequentially un-compact set  in $E$ and let $(y_n)\subset T(A)$, then $y_{n}=T(x_n)$ for some sequence $(x_n)$ in $A$. Now, as $A$ is relatively sequentially un-compact, it follows that $(x_n)$ has a subsequence $(x_{n_k})$ which is un-convergent to some $x\in E$. Using Theorem 3.2 \cite{DOT} there exist a subsequence $(x_{n_{k''}})$ of $(x_{n_k})$ and a bounded disjoint sequence $(g_{k''})$ in $E$ such that $x_{n_{k''}}-g_{k''}\overset{||.||}{\longrightarrow} 0$, on the other hand since $T$ is M-weakly compact, we have $T(g_{k''})\overset{||.||}{\longrightarrow} 0$, we infer that $T(x_{n_{k''}})\overset{||.||}{\longrightarrow} 0$, therefore $T(A)$ is a relatively compact set of $Y$.
\end{proof}

As a consequence of the above Proposition, we have the next result:

\begin{proposition} \label{MDP}
	Let  $X$, $Y$ be  Banach spaces and $F$ be a Banach lattice and let $S_1:X\longrightarrow F$ and $ S_2:F\longrightarrow Y$ be  operators, where  $S_1$ is $\sigma$-un-Dunford-Pettis and $S_2$ is M-weakly compact, then the product operator $S_2\circ S_1$ is a Dunford-Pettis operator.
\end{proposition}

\begin{proof}
	Let $A$ be a relatively weakly compact set in $X$, since  $S_{1}$ is $\sigma$-un-Dunford-Pettis it follows from Proposition \ref{DEF} that $S_{1}(A)$ is a relatively  sequentially un-compact set in $F$. Now, as $S_2$ is M-weakly compact it follows from Proposition \ref{MW} that $S_2\circ S_1(A)$ is a relatively compact set in $Y$ and hence $S_2\circ S_1$ is a Dunford-Pettis operator.
\end{proof}

The following result is a simple consequence of the above Proposition;
	
\begin{corollary}
	Let $E$ and $F$ be Banach lattices such that $E$ is atomic and order continuous. Then, each M-weakly compact operator $T:E\longrightarrow F$ is Dunford-Pettis.
\end{corollary}	

\begin{proof}
	Follows from  Proposition  \ref{DEF} and Proposition \ref{MDP}.
\end{proof}
The class of $\sigma$-un-Dunford-Pettis operators does not satisfy the duality property. That is, there exists $\sigma$-un-Dunford-Pettis operator whose adjoint is not $\sigma$-un-Dunford-Pettis. In fact, the identity operator $Id_{l_1}$ of the Banach lattice $\ell^{1}$ is $\sigma$-un-Dunford-Pettis, but its adjoint which is the identity operator $Id_{\ell^{\infty}}$ of the Banach lattice $\ell^{\infty}$ is not $\sigma$-un-Dunford-Pettis.

In the following result, we give sufficient and necessary conditions under which the direct duality of $\sigma$-un-Dunford-Pettis operators is satisfied.

\begin{theorem}
	Let $E$ and $F$ be two Banach lattices such that $E'$ is atomic. Then, the following statements are equivalent:
	\begin{enumerate}
        \item The adjoint of each $\sigma$-un-Dunford-Pettis operator $T:E\longrightarrow F$ is $\sigma$-un-Dunford-Pettis.
		\item $E'$ is order continuous or $F'$ has the Schur property.
	\end{enumerate}
\end{theorem}

\begin{proof}
	$(2)\Longrightarrow (1)$ If  $F'$ has the Schur property, then the operator $T':F'\longrightarrow E'$ is Dunford-Pettis and hence it is  $\sigma$-un-Dunford-Pettis. On the other hand, if  $E'$ is order continuous and atomic, by Proposition 2.6 \cite{DOT} and Proposition \ref{4.2} we infer that $T':F'\longrightarrow E'$ is a $\sigma$-un-Dunford-Pettis operator.\\
	$(1)\Longrightarrow (2)$ Assume that $E'$ is not order continuous and $F'$ does not satisfy the Schur property, then there exist an order bounded disjoint sequence $(g_n)$ in $E'$ such that $\| g_n\|=1$  for all $n\in\mathbb{N}$ and $ |g_n|\leq g$ for some $g\in (E')^{+}$, also there exists $f_n{\overset{w}{\longrightarrow}} 0$ in $F'$ such $\| f_n\|=1$ for each  $ n $, so we can find  $(y_n)\in B_F$ such that $|f_{n}(y_n)|> 1/2$ for each $n \in \mathbb{N}$. Now, we consider the following operators,
	$$ \begin{array}{lrcl}
	S_1 : & \ell^1 & \longrightarrow & F \\
	& (\lambda_n)_n & \longmapsto & \sum_{n=1}^{\infty} \lambda_{n}y_n  \end{array}  $$
	and
	$$ \begin{array}{lrcl}
	S_2 : & E & \longrightarrow & \ell^1 \\
	& x & \longmapsto & (g_{n}(x))  \end{array}  $$
	$S_1$ and $S_2$ are well-defined operators, and the product operator $T=S_1\circ S_2$  is $\sigma$-un-Dunford-Pettis, but its adjoint operator is not $\sigma$-un-Dunford-Pettis. Indeed, let $\phi=\frac{1}{2} g \in (E')^+$, we have
\begin{eqnarray*}
|T'(f_n)|\wedge\phi &=& |\sum_{n=1}^{\infty} f_{n}(y_n)g_n |\wedge \phi \\
   &=& |\bigvee_{n=1}^{\infty}f_{n}(y_n)g_n |\wedge \phi \\
   &\geq& |f_{n}(y_n)||g_n |\wedge \phi \\
   &\geq & \frac{1}{2}|g_n |\wedge\frac{1}{2} g = \frac{1}{2}|g_n |
\end{eqnarray*}
 then $$\| |T'(f_n)|\wedge\phi \|\geq  \frac{1}{2}\|g_n\|=\frac{1}{2},$$  this show that $T'(f_n){\overset{un}{\nrightarrow}} 0$ and hence $T'$ is not $\sigma$-un-Dunford-Pettis operator, this makes a contradiction and ends the proof.
\end{proof}


Also, the Reciprocal duality is not satisfied in the class of $\sigma$-un-Dunford-Pettis operators. Indeed, the canonical injection $ i:c_0\longrightarrow \ell^\infty$  is not $\sigma$-un-Dunford-Pettis, however its adjoint operator $ i':(\ell^\infty)'\longrightarrow \ell^1$ is $\sigma$-un-Dunford-Pettis (since it is a Dunford-Pettis operator).

\begin{theorem}
	Let $E$ and $F$ be Banach lattices such that the lattice operation of $E$ are weakly sequentially continuous and $F$ is atomic. Then the following statements are equivalent:
	\begin{enumerate}
		\item  Each operator $T:E\longrightarrow F$ is $\sigma$-un-Dunford-Pettis whenever its adjoint operator $T':F'\longrightarrow E'$ is $\sigma$-un-Dunford-Pettis.
		\item One of the following conditions is valid:
		\begin{enumerate}
			\item $E$ has the positive Schur property;
			\item $F$ is order continuous.
		\end{enumerate}
	\end{enumerate}
\end{theorem}
\begin{proof}
$(2-a)\Longrightarrow (1)$	If $E$ has the positive Schur property and the lattice operations in $E$ are weakly sequentially continuous, then $E$ has Schur property and hence $T:E\longrightarrow F$ is a $\sigma$-un-Dunford-Pettis operator.\\
	$(2-b)\Longrightarrow (1)$ Since $F$ is atomic and order continuous, it follows from Proposition 2.6 \cite{DOT} and Proposition \ref{4.2} that $T$ is $\sigma$-un-Dunford-Pettis. \\
	$(1)\Longrightarrow (2)$ Assume that $F$ is not order continuous and $E$ does not have the positive Schur property, then by Theorem 2.4.2 \cite{Mey} there exists an order bounded disjoint sequence $(y_n)$ in $F^{+}$ such that $\|y_n\|=1$   for all $n\in\mathbb{N}$ and there is an element $y\in F_{+}$ such that $0 \leq y_n\leq y $.  We consider  the operator $$ \begin{array}{lrcl}
	R : & c_0 & \longrightarrow & F \\
	& (\lambda_n)_n & \longmapsto & \sum_{n=1}^{\infty} \lambda_{n}y_n , \end{array}  $$
	 using the some arguments from the proof of Theorem 117.1 \cite{zan},  we can check that the operator $R$ is will defined. On the other hand, since $E$ does not have the positive Schur property, there exists  a disjoint weakly null sequence $(x_n$) in $E^{+}$ such that $\|x_n\|=1$ for all $n$, hence  by Lemma 3.4 \cite{hmich} there exists a positive disjoint sequence $(g_n)$ in $(E')^{+}$ with $\|g_n\|\leq 1$ such that $g_n(x_n) = 1$ for all $n$ and $g_n(x_m)= 0 $ if $n\neq m$. We  Consider the operator$$
	\begin{array}{lrcl}
	S : & E & \longrightarrow & c_0 \\
	& x & \longmapsto & (g_n(x))_{n=0}^{\infty},  \end{array}  $$

	clearly that $S$ is well defined. \\
Let $T=R\circ S$. we have

\begin{eqnarray*}
  \||T(x_n)|\wedge y\| &=& \||R\circ S(x_n)|\wedge y\|=\||R(e_n)|\wedge y\| \\
   &=& \||y_n|\wedge y \|=\|y_n\|=1 \quad \forall n\in\mathbb{N}.
\end{eqnarray*}
 Therefore $T$ fails to be $\sigma$-un-Dunford-Pettis. However, its adjoint operator $T'$ is $\sigma$-un-Dunford-Pettis.
\end{proof}

For the next results, we need to recall the definition of sequentially un-compact operators.

\begin{definition} \cite{KD}
Let $X$ be a Banach space and  $F$ be a Banach lattice.  An operator $T:X\longrightarrow F$ is said to be sequentially un-compact if $T(B_{X})$ is a relatively sequentially un-compact set of $F$, equivalently, the image of every bounded sequence $(x_n)$ in $X$ has a subsequence which is un-convergent.
\end{definition}

\begin{proposition} \label{10.1}
Each sequentially un-compact operator $T:X\longrightarrow F$ from a Banach space $X$  into a Banach lattice $F$ is $\sigma$-un-Dunford-Pettis.
\end{proposition}

\begin{proof}
	Let $A\subset X$ be a weakly compact set, hence $A$ is a bounded set of $X$. On the other hand, since $T$ is sequentially un-compact we have $T(A)$ is a relatively  sequentially un-compact set of $F$. Thus, $T$ is $\sigma$-un-Dunford-Pettis operator.
\end{proof}

Note that the converse of the previous result does not hold. Indeed, if we consider $Id_{c_0}$ the identity operator of $c_0$, it is clear that $Id_{c_0}$ is $\sigma$-un-Dunford-Pettis (since $c_0$ is atomic and order continuous), but fails to be sequentially un-compact (since $c_0$ it is not KB-space (See Theorem 7.5 \cite{KD})).

From Proposition  \ref{DEF} we have the following result;

\begin{corollary} \label{13.1}
	If $X$ is a reflexive Banach lattice then for every Banach lattice $F$, each $\sigma$-un-Dunford-Pettis operator $T:X\longrightarrow F$	
	is sequentially un-compact.	
\end{corollary}

In the following result, we characterize Banach lattices under which each $\sigma$-un-Dunford-Pettis operator is sequentially un-compact.
\begin{proposition}
	Let $E$ and $F$ be Banach lattices such that $E$ is a KB-space. The following statements are equivalent:
	\begin{enumerate}
		\item Each $\sigma$-un-Dunford-Pettis operator $T:X\longrightarrow F$ is sequentially un-compact.
		\item One of the conditions is holds:
        \begin{enumerate}
            \item $E$ reflexive;
			\item $F$ is an atomic KB-space.
		\end{enumerate}
	\end{enumerate}
\end{proposition}

\begin{proof}
	 $(2-b)\Longrightarrow (1)$ Assume that $F$ is an atomic KB-space, we infer  by Proposition 9.1 \cite{KD} that  $T:X\longrightarrow F$	is a sequentially un-compact operator.\\
	$(2-a)\Longrightarrow (1)$ Follows from Corollary \ref{13.1}. \\
	$(1)\Longrightarrow (2)$ Suppose that (2) does not hold, then neither $F$
	is an atomic KB-space nor $X$ is reflexive, it follows from Theorem 7.5 \cite{KD}  that there is a bounded sequence $(y_n)$ in $F$ which does not have any un-convergent subsequence.\\ We may consider the operator;  $$ \begin{array}{lrcl}
	S : & \ell^1 & \longrightarrow & F \\
	& (\lambda_n)_n & \longmapsto & \sum_{n=1}^{\infty} \lambda_{n}y_n  \end{array}  $$
	which is clearly $\sigma$-un-Dunford-Pettis. On the other hand, by Theorem 2.4.15 \cite{Mey} $E'$ is order continuous and it follows from Theorem 2.4.14 \cite{Mey} that $E$ contains a sublattice isomorphic  copy of $\ell^1$ and there exist a positive projection $P:E\longrightarrow \ell^1$ and a canonical injection $i:\ell^1\longrightarrow E$. Now, if we consider the composed operator $T=S\circ P:E\longrightarrow \ell^1 \longrightarrow F$,  $T$ is clearly  $\sigma$-un-Dunford-Pettis and by our hypothesis  $T$ will be sequentially un-compact, so $T\circ i=S\circ P\circ i=S$ is sequentially un-compact, thus $S(e_n)=y_n$ has an un-convergent subsequence, this makes a contradiction and completes the proof.
\end{proof}

\end{document}